\documentclass[]{article}

\usepackage[T1]{fontenc}
\usepackage[cp1250]{inputenc}

\usepackage{amsmath}
\usepackage{amsfonts}
\usepackage{amssymb}
\usepackage{amsthm}
\usepackage{amscd}
\usepackage{wasysym}
\input xy
\xyoption{all}

\newtheorem{theo}{Theorem}[section]
\newtheorem{prop}{Proposition}[section]
\newtheorem{lem}{Lemma}[section]
\newtheorem{cor}{Corollary}[section]

\theoremstyle{definition}
\newtheorem{example}{Example}[section]
\newtheorem{definition}{Definition}[section]
\newtheorem{remark}{Remark}[section]

\newcommand{\N}{\mathbb N}

\newcommand{\Csi}{\mathcal C^{\infty}}
\newcommand{\Real}{\mathbb R}
\newcommand{\dd}{\mathbf d}
\newcommand{\As}{\mathcal A}
\def\jetbndln{T^n M}
\def\jetbndl#1{T^{#1} M}
\def\jetb#1#2{T^{#1} #2}
\def\jetspacen#1#2#3{T^{#1}_{#2}#3}

\def\jet#1{[#1]}
\def\kdiff#1{\frac{1}{{#1}!}\left. \frac{d^{#1}}{dt^{#1}}\right| _{t=0}}
\def\diff1{\left. \frac{d}{dt}\right| _{t=0}}
\def\jetbndlmn#1{T^{m,n} #1}
\def\xars{x_{(r,s),a}}
\def\differ{T}

\font\black=cmbx10 \font\sblack=cmbx7 \font\ssblack=cmbx5 \font\blackital=cmmib10 \skewchar\blackital='177
\font\sblackital=cmmib7 \skewchar\sblackital='177 \font\ssblackital=cmmib5 \skewchar\ssblackital='177
\font\sanss=cmss10 \font\ssanss=cmss8 scaled 900 \font\sssanss=cmss8 scaled 600 \font\blackboard=msbm10
\font\sblackboard=msbm7 \font\ssblackboard=msbm5 \font\caligr=eusm10 \font\scaligr=eusm7 \font\sscaligr=eusm5
 \font\fraktur=eufm10 \font\sfraktur=eufm7 \font\ssfraktur=eufm5 

\font\bsymb=cmsy10 scaled\magstep2
\def\all#1{\setbox0=\hbox{\lower1.5pt\hbox{\bsymb
       \char"38}}\setbox1=\hbox{$_{#1}$} \box0\lower2pt\box1\;}
\def\exi#1{\setbox0=\hbox{\lower1.5pt\hbox{\bsymb \char"39}}
       \setbox1=\hbox{$_{#1}$} \box0\lower2pt\box1\;}

\def\tx#1{{\fam0\relax#1}}

\newfam\bifam
\textfont\bifam=\blackital \scriptfont\bifam=\sblackital \scriptscriptfont\bifam=\ssblackital

\newfam\blfam
\textfont\blfam=\black \scriptfont\blfam=\sblack \scriptscriptfont\blfam=\ssblack

\newfam\bbfam
\textfont\bbfam=\blackboard \scriptfont\bbfam=\sblackboard \scriptscriptfont\bbfam=\ssblackboard

\newfam\ssfam
\textfont\ssfam=\sanss \scriptfont\ssfam=\ssanss \scriptscriptfont\ssfam=\sssanss

\newfam\clfam
\textfont\clfam=\caligr \scriptfont\clfam=\scaligr \scriptscriptfont\clfam=\sscaligr

\newfam\frfam
\textfont\frfam=\fraktur \scriptfont\frfam=\sfraktur \scriptscriptfont\frfam=\ssfraktur

\def\hpb#1{\setbox0=\hbox{${#1}$}
    \copy0 \kern-\wd0 \kern.2pt \box0}
\def\vpb#1{\setbox0=\hbox{${#1}$}
    \copy0 \kern-\wd0 \raise.08pt \box0}

\def\pmb#1{\setbox0\hbox{${#1}$} \copy0 \kern-\wd0 \kern.2pt \box0}
\def\pmbb#1{\setbox0\hbox{${#1}$} \copy0 \kern-\wd0
      \kern.2pt \copy0 \kern-\wd0 \kern.2pt \box0}
\def\pmbbb#1{\setbox0\hbox{${#1}$} \copy0 \kern-\wd0
      \kern.2pt \copy0 \kern-\wd0 \kern.2pt
    \copy0 \kern-\wd0 \kern.2pt \box0}
\def\pmxb#1{\setbox0\hbox{${#1}$} \copy0 \kern-\wd0
      \kern.2pt \copy0 \kern-\wd0 \kern.2pt
      \copy0 \kern-\wd0 \kern.2pt \copy0 \kern-\wd0 \kern.2pt \box0}
\def\pmxbb#1{\setbox0\hbox{${#1}$} \copy0 \kern-\wd0 \kern.2pt
      \copy0 \kern-\wd0 \kern.2pt
      \copy0 \kern-\wd0 \kern.2pt \copy0 \kern-\wd0 \kern.2pt
      \copy0 \kern-\wd0 \kern.2pt \box0}

\def\xd{\tx{d}}
\def\xi{\tx{i}}

\newcommand{\ee}{\end{equation}}
\newcommand{\ra}{\rightarrow}

\newcommand{\bea}{\begin{eqnarray}}
\newcommand{\eea}{\end{eqnarray}}
\newcommand{\beas}{\begin{eqnarray*}}
\newcommand{\eeas}{\end{eqnarray*}}

\newcommand{\R}{\mathbb{R}}

\newcommand{\GL}{\operatorname{GL}}
\newcommand{\GH}{\operatorname{GG}}

\newcommand{\nn}{\nonumber}

\newcommand{\pa}{\partial}
\newcommand{\ti}{\times}

\mathchardef\za="710B  %\alpha
\mathchardef\zb="710C  %\beta
\mathchardef\zg="710D  %\gamma
\mathchardef\zd="710E  %\delta
\mathchardef\zve="710F %\epsilon
\mathchardef\zz="7110  %\zeta
\mathchardef\zh="7111  %\eta
\mathchardef\zvy="7112 %\theta
\mathchardef\zi="7113  %\iota
\mathchardef\zk="7114  %\kappa
\mathchardef\zl="7115  %\lambda
\mathchardef\zm="7116  %\mu
\mathchardef\zn="7117  %\nu
\mathchardef\zx="7118  %\xi
\mathchardef\zp="7119  %\pi
\mathchardef\zr="711A  %\rho
\mathchardef\zs="711B  %\sigma
\mathchardef\zt="711C  %\tau
\mathchardef\zu="711D  %\upsilon
\mathchardef\zvf="711E %\phi
\mathchardef\zq="711F  %\chi
\mathchardef\zc="7120  %\psi
\mathchardef\zw="7121  %\omega
\mathchardef\ze="7122  %\varepsilon
\mathchardef\zy="7123  %\vartheta
\mathchardef\zf="7124  %\varomega
\mathchardef\zvr="7125 %\varrho
\mathchardef\zvs="7126 %\varsigma
\mathchardef\zf="7127  %\varphi
\mathchardef\zG="7000  %\Gamma
\mathchardef\zD="7001  %\Delta
\mathchardef\zY="7002  %\Theta
\mathchardef\zL="7003  %\Lambda
\mathchardef\zX="7004  %\Xi
\mathchardef\zP="7005  %\Pi
\mathchardef\zS="7006  %\Sigma
\mathchardef\zU="7007  %\Upsilon
\mathchardef\zF="7008  %\Phi
\mathchardef\zW="700A  %\Omega

\newcommand{\epf}{\hfill$\Box$}
\newcommand{\bepf}{\noindent\textit{Proof.-} }

\def\be#1{\begin{equation}\label{#1}}
\begin{document}
\title{Graded bundles and homogeneity structures
\thanks{Research
supported by the Polish Ministry of Science and Higher Education under the grant N N201 416839.} }

        \author{
        Janusz Grabowski$^1$, Miko\l aj Rotkiewicz$^2$\\
        \\
         $^1$ {\it Institute of Mathematics}\\
                {\it Polish Academy of Sciences}\\\\
         $^2$ {\it Institute of Mathematics}\\
                {\it University of Warsaw}
                }
\date{}
\maketitle

\begin{abstract} We introduce the concept of a {\it graded bundle} which is a natural generalization of the concept of  a vector bundle and whose standard examples are higher tangent bundles $T^n Q$ playing a fundamental role in higher order Lagrangian formalisms. Graded bundles are graded manifolds in the sense that we can choose an atlas whose local coordinates are homogeneous functions of degrees $0,1,\dots,n$. We prove that graded bundles have a convenient equivalent description as {\it homogeneity structures}, i.e. manifolds with a smooth action of the multiplicative monoid $(\R_{\ge 0},\cdot)$ of non-negative reals. The main result states that each homogeneity structure admits an atlas whose local coordinates are homogeneous.
Considering a natural compatibility condition of homogeneity structures we formulate, in turn, the concept of {\it double}
({\it $r$-tuple}, in general) {\it graded bundle} --  a broad generalization of the concept of double
($r$-tuple) vector bundle. Double graded bundles are proven to be locally trivial in the sense that we
can find local coordinates which are simultaneously homogeneous with respect to both homogeneity structures.


\end{abstract}

\bigskip

\section{Introduction}
Starting from the well-known observation that differentiable 1-homogeneous functions on $\R^N$ are
automatically linear, we provided in \cite{GR} an easy and effective characterization of these smooth actions
of the monoid $(\R_{\ge 0},\cdot)$ of multiplicative non-negative reals on a manifold $M$ which come from
homotheties of vector bundle structures on $M$. We obtained the vector bundle structure by an identification
of $M$ with a vector subbundle of $TM$. All this, in turn, allowed us to describe several concepts of the
theory of vector bundles purely in terms of the {\it homogeneity structures} on vector bundles, defined by the
corresponding homotheties.

For instance, a vector bundle morphism is just a smooth map which intertwines the homotheties, and a vector
subbundle turns out to be just a submanifold which is invariant with respect to the homotheties. What is more,
while it is not easy to explain a compatibility between additive structures, this question becomes nearly
obvious in the language of homogeneity: two vector bundle structures on a manifold are compatible if the
corresponding homotheties commute. This leads to an elegant and effective definition of double (or $n$-tuple)
vector bundle.

The language of homogeneity can be also easily adapted in the case of supergeometry, allowing us to associate
with any $n$-tuple vector bundle $E$ an $\N^n$-graded supermanifold, the `superization' of $E$. This is a natural
generalization of the well-known procedure $E\mapsto \Pi E$ of reversing parity in the fibers of a vector
bundle $E$.

All this suggests that studying homogeneity structures which are more general than those induced by polynomial rings
generated only by 0- and 1-homogeneous functions, like in the case of vector bundles, can be of great interest and can provide us with a useful tool for dealing with other interesting geometric categories of fibrations.

We start this program in the present paper, where a {\em graded bundle of degree} $n$ is defined as a
fibration $\zp:M\ra M_0$ whose fibers are consistently identified with $\R^N$ equipped with a structure of
{\it graded space of degree $n$}. The latter assigns to canonical coordinates in $\R$ their {\it degrees}
taking values in $\{ 1,\dots, n\}$. Consequently, a graded bundle $M$ of degree $n$ possesses an atlas
whose local coordinates have integer degrees between 0 and $n$, compatible with the changes of local coordinates.
Moreover, it admits a canonical {\it homogeneity structure} defined by the obvious action $h:[0,\infty)\times
M\ra M$ of the multiplicative semigroup $(\R_{\ge 0},\cdot)$ in which $h_t=h(t,\cdot)$ maps the coordinate $x$ of degree $k$ into $t^kx$.
In particular, $h_0$ is just the fibration projection.

It makes sense to speak about global homogeneous functions on $M$ and the corresponding polynomial algebra, but
we must stress that a graded bundle is not just a manifold with consistently defined homogeneity of local
coordinates. What we require additionally is equivalent to the fact that the natural {\it weight vector field}
encoding the homogeneity is complete. One can see easily the difference, comparing $\R^N$ with the natural
homogeneity structure, in which all linear functions have degree 1, with an open disc in $\R^N$ having the
coordinates inherited from $\R^N$. It makes sense to speak about homogeneity of functions on the disc, but the
disc is not a graded space in our sense, as it is not complete.

Our fundamental examples of graded bundles of degree $n$ are the $n$th tangent bundles $T^nM_0$, i.e. the
bundles  $J^n_0(\R,M_0)$ of $n$th jets of curves in $M_0$. The higher tangent bundles have been extensively studied, mainly in relation to higher order Lagrangian formalisms \cite{CSC,dLR,Mi1,Mi2,Tul0,YI}. Also
$n$-vector bundles, e.g. the double vector bundles $TTM_0$ or $T^*TM_0\simeq TT^*M_0$, are canonically
graded bundles.

In the present paper, we find, like in the case of the vector bundles \cite{GR}, a characterization of these homogeneity structures on a manifold $M$ (i.e. those actions $h$ of the multiplicative semigroup $(\R_{\ge 0},\cdot)$) which come from the structure of a graded bundle of degree $n$. The necessary and sufficient condition for $h$ tells us that the $n$th jets of the curves $t\mapsto h(t,p)$ at 0 vanish only for $p\in M_0$. We call such actions {\it homogeneity structures
of degree $n$} and our main result says, roughly speaking, that for homogeneity structures of degree $n$ on a
manifold $M$ we can find an atlas whose local charts consist of homogeneous coordinates. We show that homogeneity structures can be equivalently characterized as manifolds equipped with a smooth action of the monoid $(\R_{\ge 0},\cdot)$. In other words, if $h(1,\cdot)=id_M$, the nondegeneracy of jets follows automatically. Note also that we can use only non-negative reals, as their action can be uniquely extended to an action of $(\R,\cdot)$.

One should remark that graded bundles have already appeared in the supergeometry, where {\em N-manifolds of
degree} $n$ have been studied by {\v{S}}evera and Roytenberg \cite{Sev, Roy}, and applied in the theory of
Courant algebroids and Dirac structures. N-manifolds are exactly supermanifolds with local coordinates of
degrees between 0 and $n$ whose parity coincides with the coordinate degree parity. The equivalence with homogeneity structures has been claimed in this case as well; however, some problems in the
supergeometric case are simpler, as odd coordinates are always `linear' and `complete'. We use the term `graded bundle' rather than `graded manifold' to distinguish our approach from those based on various different concepts of graded manifolds, usually associated with a supermanifold structure (see the discussion in \cite{Vo}).

Using a natural concept of compatibility of homogeneity structures, we derive also the concept of a double
(or, more generally, $r$-tuple) homogeneity structure (or graded bundle) as consisting of a manifold $F$ with two commuting homogeneity structures $h^1,h^2$, $h^1_t\circ h^2_u=h^2_u\circ h^1_t$. The main result in this direction says
that double homogeneity structures are locally trivial in a natural sense: we can find local coordinates which
are simultaneously homogeneous with respect to both homogeneity structures. Note also that double homogeneity
structures, unlike double vector bundle structures, generate a new homogeneity structure $h$ defined by
$h_t=h^1_t\circ h^2_t$.

There are, of course, natural questions concerning the concepts of duality for homogeneity structures and
their applications in physics, investigated recently by Tulczyjew \cite{Tul} in the context of the higher tangent
bundles (see also \cite{Mi2}), which we decided, however, to postpone to a separate paper.

\section{Graded spaces}

Let us start with simple observations which, however, will explain the motivation for more advanced concepts
introduced further in the paper.

The vector space $\R^N$ with canonical coordinates $y=(y^1,\dots,y^N)$ has naturally defined homotheties
$h_t$, $t\in \R$, related to the multiplication by reals in $\R^N$,
$$h_t(y)=t. y=(t\, y^1,\dots,t\, y^N)\,.$$
As $h_t\circ h_s=h_{ts}$, the homotheties define an action of the multiplicative semigroup $(\R,\cdot)$, and so
its sub-semigroup $(\R_{\ge 0}, \cdot)$ of non-negative reals. Inside $(\R_{\ge 0}, \cdot)$ we have the
one-parameter multiplicative group $(\R_{>0}, \cdot)$ of positive reals, canonically isomorphic, according to
the map $\R\ni t\mapsto e^t\in\R_{>0}$, to the group of additive reals $(\R,+)$, so that $h_{e^t}$, $t\in\R$,
is a one-parameter group of diffeomorphisms of $\R^N$ with the generator called the {\it Euler vector field}
$$\zD(y)=\left.\frac{d}{dt}\right| _{t=1}h_t(y)=\sum_iy^i\pa_{y^i}\,,$$
i.e., $h_{e^t}=\exp{(t\zD)}$. Note that the Euler vector field uniquely determines
$h_t$ for $t\ge 0$, as $$h_0=\lim_{s\to-\infty}\exp{(s\zD)}\,.$$ The presence of the semigroup action
$h:\R_{\ge 0}\times\R^N\to\R^N$, $h(t,y)=h_t(y)$, allows us to define {\it homogeneous functions of degree
$r$} on $\R^N$, that we write $\deg(f)=r$, as functions $f:\R^N\to\R$ satisfying
\be{hf}
f\circ h_t=t^rf,\quad \text{for all} \ t> 0\,.
\ee
In terms of the Euler vector field, (\ref{hf}) reads
\be{hf1}
\zD(f)=rf\,.
\ee
It follows  that the algebra $\Csi(\R^N)$ of smooth functions on $\R^N$ contains a distinguished graded
subalgebra
$$\As(\R^N) = \bigoplus_{k=0}^\infty \As_k(\R^N)
$$
of {\it polynomial functions}, where $\As_k(\R^N)$ is the space of homogeneous polynomials of degree $k$
on $\R^N$. According to Euler's Homogeneous Function Theorem, homogeneous functions of degree 1 are linear,
thus the whole structure of the vector space is actually encoded in $h$ (or in $\zD$). In particular, any
diffeomorphism $\zf:\R^N\to\R^N$ is linear if and only if it induces an isomorphism of the graded algebra
$\As(\R^N)$ of polynomial functions, if and only if it intertwines the action $h$, $\zf\circ h_t=h_t\circ\zf$,
and if and only if it respects the Euler vector field, $\zf_*(\zD)=\zD$. Hence, the homogeneity of $\R^N$ can
be equivalently described in terms of either $\As(\R^N)$, or $h$, or $\zD$. The same can actually be done in
the case of any finite-dimensional real vector space $V$ replacing $\R^N$.

Since we want to extend this model and, in the simplest situation, to allow coordinates in $\R^N$ to have
various degrees of homogeneity, we propose the following.

\begin{definition}
A {\it standard homogeneity structure of degree $n$ and rank $\dd =  (d_1, \ldots, d_n)$} on $\R^N$, where $N
=\sum_{i=1}^n d_i$, is the action $h : \R_{\ge 0} \times \R^N \to \R^N$ of the semigroup $(\R_{\ge 0}, \cdot)$
%determined by
%an $N$-tuple of degrees $(w_1, \ldots, w_N)$, $w_i \in\{ 1,\dots,n\}$ for
%$1\leq i\leq N$,
which in the canonical coordinates $(y^i)$ of $\R^N$ reads
\be{hf2} h_t(y^1,\dots,y^N)=(t^{w_1}y^1,\dots,t^{w_N}y^N),\, \ee
where $h_t = h(t, \cdot)$ and $w_i = j$ for $d_1+\ldots+d_{j-1} + 1\leq i\leq d_1+\ldots+d_j$. In other words,
$d_i$ is the number of coordinates of degree $i$. The space $\R^N$ equipped with the standard homogeneity
structure of rank $\dd$ we shall denote with $\R^\dd$ and call the {\it standard graded space of rank
$\dd$}. Of course, one can identify $\R^\dd$ also by declaring the homogeneity degrees $w_i$ of the coordinate
$y^i$ for all $1\le i\le N$.
\end{definition}
\begin{definition} A {\it graded space of degree $n$} and rank $\dd =  (d_1, \ldots, d_n)$,
where $N =\sum_{i=1}^n d_i$, is a smooth manifold $M$ equipped with a smooth action $h$ of $(\R_{\ge 0},
\cdot)$ for which there exists a diffeomorphism $\zf : M\ra\R^\dd$ onto $\R^N$ equipped with the standard
homogeneity structure of rank $\dd$, intertwining the actions of $(\R_{\ge 0}, \cdot)$. We call the maps $h_t$ the {\it homotheties} of the graded space $M$ .

A smooth function $f$ on a homogeneity  space $(M, h)$ of degree $n$ is called {\it homogeneous of degree
$r$}, which we write as $\deg(f)=r$, if (\ref{hf}) is satisfied. Hence, the canonical coordinate $y^i$ on
$\R^\dd$, as well as the smooth function $y^i\circ \zf$ on $M$, denoted with some abuse of notation also
as $y^i$, has the degree $1\le w_i\le n$. We shall call the (global) functions $(y^i)$ on $M$ {\it homogeneous
coordinates} and $w=(w_1,\dots,w_N)$ the {\it weight vector}. Equivalently, one can choose
global coordinates $(y^1,\dots,y^N)$ establishing a diffeomorphism onto $\R^N$ and declare $y^i$ to be
homogeneous of degree $w_i$.
A morphism of graded spaces is a smooth map $\psi:M_1\ra M_2$ that intertwines the actions of $(\R_{\ge
0}, \cdot)$.
%We shall show (Corollary \ref{cor:auto}) that it is equivalent to say that $\psi$
%preserves the homogeneous degree, i.e. $f\circ\psi$ is homogeneous
%of degree $r$ if $f$ is so.
\end{definition}

Like in the case of a vector space, we can define the polynomial algebra
$$\As(M) = \bigoplus_{k=0}^\infty \As_k(M)\,,
$$
where $\As_k(M)$ is the space of smooth homogeneous functions of degree $k$ on $M$.

Note that $h_1=id_M$, so the action (\ref{hf2}) is actually a monoid action of $(\R_{\ge 0},\cdot)$ and that
the rank $\dd$ (or the weight vector $w$) of a graded space is uniquely determined. Indeed, we have $h_t\circ h_s=h_{ts}$ and
$w_iy^i=\left. \frac{d}{dt}\right| _{t=1}\left( y^i\circ h_t\right)$. We can therefore associate with the graded space $(M,h)$ the {\it weight vector field} $\Delta_M$ on $M$, $\Delta_M(p) =  \left. \frac{d}{dt}\right|
_{t=1} h_t(p)$, which in homogeneous coordinates reads $$\Delta_M=\sum_{i=1}^Nw_iy^i\partial_{y^i}\,.$$ Since
$w_i>0$, the weight vector field is complete and if $t\mapsto \exp{(t\zD_M)}$ is the flow of diffeomorphisms
it generates, then we have $\exp{(t\zD_M)}=h_{e^t}$ for all $t\in\R$. Thus, the action by homotheties is
completely determined by the weight vector field and {\it vice versa}. Moreover, the constants $w_i$ are
invariants of the orbit of the vector field $\zD_M$ under diffeomorphisms. This easily implies the following.
\begin{theo} Let $M_a$ be a graded space with homotheties $(h_t^a)$, $a=1,2$, and $\psi:M_1\ra M_2$ be a smooth map.
The following are equivalent:
\begin{description}
\item{(a)}  The map $\psi$ is a morphism of graded spaces, i.e. $\psi\circ h^1_t=h^2_t\circ\psi$ for $t\ge 0$; \item{(b)}  For each smooth homogeneous
function $f:M_2\ra\R$ of degree $r$, the function $f\circ\psi$ is homogeneous of degree $r$, i.e.,
$\psi^*(\As_r(M_2))\subset\As_r(M_1)$, $r=1,2,\dots$; \item{(c)}  The map $\psi$ relates the weight vector fields $\zD_{M_1}$
and $\zD_{M_1}$, $\psi_*(\zD_{M_1})\subset\zD_{M_2}$.
\end{description}
In particular, two graded spaces are isomorphic if and only if they have the same rank.
\end{theo}

\begin{remark}
We must stress that requiring the global diffeomorphism with the standard graded space $\R^\dd$ is very
important for having the action $h$ properly defined; it is not enough to assume some degrees (weights) for some
coordinates in $M$. Indeed, we can cut the open unit disc from $\R^N$ with the same coordinates and degrees,
but we do not get a graded space in this way, as the map $h_t$ is not defined in the disc for $t>1$ (the
weight vector field restricted to the disc is not complete).
\end{remark}

We have the following characterization of homogeneous functions on a graded space which shows that they
belong to the polynomial algebra $\R[y^1, \ldots, y^N]$ considered as a graded algebra with the $\N$-gradation
induced by degrees of $y^1, \ldots, y^N$. Note that, according to (\ref{hf}),  a polynomial $f\in \R[y^1,
\ldots, y^N]$ belonging to a component of gradation $k\in \N$, when considered as a function on a graded space, also has degree $k$. The following lemma can be viewed as a generalization of the Euler's Homogeneous
Function Theorem.
\begin{lem}\label{L:e}
Let $M$ be a graded space of dimension $N$  with homogeneous coordinates $(y^i)$ and the weight vector
$w=(w_1,\dots,w_n)$. Then, any smooth homogeneous function on $M$ is a polynomial function in variables
$(y^i)$, $f\in\R[y^1,\dots,y^N]$. In particular, $\As(M)=\R[y^1,\dots,y^N]$ is the polynomial ring in
variables $y^1,\dots,y^N$, although equipped with the non-standard gradation implemented by fixing the degree
$w_i$ of $y^i$.
\end{lem}
\bepf We shall prove the lemma by induction with respect to $k$ -- the degree of $f$.
We can identify $M$ with $\Real^N$ being a graded space of degree $n$ with coordinates $(x_{r, a})$,
$1\leq r\leq n$, $1\leq a\leq n_r$, $\sum_r n_r = N$, where $x_{r, a}$ is of degree $r$.

Let us first notice that the only continuous $0$-homogeneous functions are constants and there are no non-zero smooth
$k$-homogeneous functions for $k<0$. Let us assume that $f$ is $k$-homogeneous, $k>0$, i.e., $f(h(t, p))=
t^k\cdot f(p)$ for all $t>0$ and $p\in \Real^N$, where
$$
h(t,x_{r, a})= h_t(x_{r,a}) = (t^r \cdot  x_{r, a})\,.
$$
Let us take $\ze\in\Real^N$,  $\ze=(0,\ldots, 0,1,0\ldots,0)$, with $x_{r,a}(\ze) = 1$. Differentiating
$f(h(t,x) + h(t, s\ze)) = t^k\,f(x+s\ze)$ with respect to $s$ at $s=0$, we get
$$
t^r \frac{\partial f}{\partial x_{r,a}} (h(t,x)) = t^k \frac{\partial f}{\partial x_{r,a}} (x).
$$
Therefore $\frac{\partial f}{\partial x_{r, a}}$ is $(k-r)$-homogeneous, so, by the inductive assumption,
$\frac{\partial f}{\partial x_{r, a}}$ is a polynomial. As all partial derivatives of $f$ are polynomials in
variables $x_{r, a}$, the same is true for $f$, that completes the proof.

\epf

 \begin{cor}\label{cor:auto} Any automorphism of a graded space $M$ is a diffeomorphism of $M$ which is polynomial
 in homogeneous coordinates, i.e., in homogeneous coordinates $(y^i)$ with the weight vector $w=(w_1,\dots,w_N)$ it takes the form
\be{homizo}\psi(y^1,\dots,y^N)=\left(\sum_{\langle w|k\rangle=w_{i}}
\za^i_k\,(y^1)^{k_1}\cdots(y^N)^{k_N}\right)_{i=1}^N\,,
\ee
where $k=(k_1,\dots,k_N)$, $\za^i_k\in\R$ and
$$\langle w|k\rangle=\sum_{j=1}^Nw_jk_j\,.
$$
\end{cor}
Since the group $\GH(\dd,M)$ of all automorphisms of a graded space $M$ (the {\it general graded group}) is parametrized by a finite number
of coefficients $\za^i_k$ as above with clear smoothness of compositions, we get the following.
\begin{cor}\label{cor:group} The group $\GH(\dd,M)$ of all automorphisms of a graded space $M$ of rank $\dd$ is a Lie group.
\end{cor}
The group of automorphism of the canonical graded space $\R^\dd$ we shall denote $\GH(\R^\dd)$. Of
course, $\GH(\dd,M)\simeq\GH(\R^\dd)$. If $\dd=(N)$, then we deal with a vector space of dimension $N$ and
$\GH(\R^\dd)=\GL(N,\R)$.
\begin{example}
To describe the group $\GH(\R^\dd)$ for $\dd = (d_1,d_2)$, consider canonical coordinates
$(x^1, \ldots, x^{d_1}; y^1, \ldots, y^{d_2})$ of $\R^\dd$. Any automorphism of
$\R^\dd$ is described by constants $a_i^j, b^u_w, c_{ij}^u\in\R$ and has the form $(x^i, y^u) \mapsto ({\bar x^i},
{\bar y^u})$ with
$$
{\bar x^i}  =\sum_j a_j^i x^j, \quad {\bar y^u}  = \sum_u b^u_w y^w +  \sum_{ij} c_{ij}^u x^i x^j.
$$
The matrices $A=(a_j^i)$ and $B=(b^u_w)$ are invertible, while $c^u_{ij}$ are arbitrary. Let $V$ be the vector subspace
of $\Csi(\R^{d_1+d_2})$ spanned by the functions $x^i, y^w$ and $z^{ij} : =x^ix^j$. It is a finite-dimensional
subspace invariant with respect to the canonical action of $\GH(\R^\dd)$ and it gives rise to a faithful
representation of $\GH(\R^\dd)$. For instance, $\GH(\R^{(1, 1)})$ can be represented as the following matrix
subgroup
$$
\GH(\R^{(1, 1)}) \simeq  \left\{ \left(
\begin{array}{ccc}
a & 0 & 0 \\
0 & b & 0 \\
0 & c  & a^2
\end{array}
\right): a, b\neq 0 \right\}.
$$
\end{example}

\section{Graded bundles}
\begin{definition} %Let $V$ be a fixed graded space of rank $\dd = (d_1, \ldots, d_n)$.
A {\it graded bundle of degree $n$ and rank $\dd$} is a smooth fiber bundle $\pi:M\ra M_0$ with the
typical fiber $\R^N$, considered as the standard homogeneity  space $\R^\dd$ of degree $n$ and rank $\dd$, which
admits an atlas of local trivializations $\phi_U:\pi^{-1}(U)\ra U\times \R^\dd$ such that the change of local
trivializations $\phi_{U'}\circ\phi_U^{-1}:(U\cap U')\times \R^\dd\ra (U\cap U')\times \R^\dd$ is in each
fiber an automorphism of the standard homogeneity structure, i.e.,
$$\phi_{U'}\circ\phi_U^{-1}(x,y)=\left(x,g_{UU'}(x)(y)\right)$$
for some $\GH(\R^\dd)$-valued {\it transition functions} $g_{UU'}$.
\end{definition}

\medskip It is clear that each fiber of a graded bundle of degree $n$ and rank $\dd$ carries the structure of a graded space
of degree $n$ and rank $\dd$. Roughly speaking, a homogeneity  bundle
is a smooth family of homogeneity  spaces parametrized by a base manifold $M_0$. If $M_0$ is a single point,
then $M$ is just a homogeneity  space. It follows  that the algebra $\Csi(M)$ of smooth functions on $M$
contains a distinguished graded subalgebra
$$\As(M) = \bigoplus_{k=0}^\infty \As_k(M)
$$
of {\it polynomial functions} on $M$, which is locally generated by $\Csi(M_0)$ and homogeneous coordinates
$y^i$ of the typical fiber $\R^\dd$. The summand $\As_k=\As_k(M)$ is the space of functions of degree $k$. In
particular, $\As_0 = \Csi(M_0)$. Local coordinates in $M$ consisting of homogeneous local functions
establishing an isomorphism with $U\ti\R^\dd$ we shall call {\it homogeneous coordinates}. The systems of homogeneous coordinates are always of
the form $(x^a,y^i)$, where $(x^a)_1^m$ are local coordinates in $M_0$ and $(y^i)_1^N$ are homogeneous
coordinates in the typical fiber. We can unify the two kinds of coordinates using local coordinates
$(y^i)_1^{N+m}$, where $y^{N+a}=x^a$, and putting the degrees $w_j=0$ for $N<j\le N+m$. We will call
$w=(w_1,\dots,w_{N+m})$ the {\it weight vector} of the graded bundle.

{\it Morphisms} of graded bundles are defined in an obvious way: they are smooth maps intertwining the
semigroup $(\Real_{\ge 0}, \cdot)$ actions, or equivalently, respecting the homogeneity of functions together
with their degrees. Also the notion of a {\it graded subbundle} of a homogeneity  bundle $M$ is clear: it
is a submanifold  defined locally by the vanishing of some homogeneous coordinates. Any graded subbundle is
clearly a graded bundle itself and we can choose as its local homogeneous coordinates the remaining part
of (non-vanishing) homogeneous local coordinates in $M$. Let us remark that a graded bundle
can be regarded as a kind of graded manifold, since it has a distinguished gradation in a dense subsheaf of the structural sheaf.

\begin{example} Graded bundles of degree $1$ correspond exactly to vector bundles. Indeed, all the degrees
$w_i$ equal one and the gluing transformations have to be linear because they preserve the degree. Moreover,
there is a unique structure of a vector space on $\Real^N$ for which $y^i$ are 1-homogeneous, thus linear,
functions.
\end{example}

\begin{remark} We can define {\it homogeneity super-bundles of degree} $n$ completely analogously, assuming only that
$\As$ is locally the polynomial super-algebra $\Csi(U)[y^1, \ldots, y^N]$ in which the parity of $y^i$ agrees
with the parity of the degree $w_i$. This notion is equivalent to the notion of an {\it N-manifold} in the
terminology of {\v{S}}evera and Roytenberg \cite{Sev,Roy}. In this sense, the concept of a homogeneity
super-bundle of degree $n$ is the even counterpart of the notion of an N-manifold of degree $n$. In general,
however, there is no obvious `superization' procedure, analogous to $\Pi$, for homogeneity  bundles of degree
$n>1$, like in the case of  vector bundles.

For explanation, let us consider an (even) homogeneity  space $V = \R^\dd$ of rank $\dd  =(2,1)$ and an
automorphism $f$ of $V$ given in canonical coordinates $x_1, x_2, y$ of degrees $1,1,2$ respectively, by
$$
f(x_1, x_2, y) = (x_1, x_2, y+x_1^2).
$$
In coordinates $x_1' = (x_1+x_2)/2$, $x_2' = (x_1-x_2)/2$ and $y'=y$ we have
$$
f^*(y') = y + x_1^2 = y' + (x_1'+x_2')^2  = y' + x_1'^2 + x_2'^2 + 2\, x_1'\cdot x_2'.
$$
Let $\Pi V$ be an $N$-manifold of degree $2$ determined by the polynomial super-algebra $\R[\zvy_1, \zvy_2,
\zt]$ in which $\zvy_1, \zvy_2$ are odd coordinates of degree $1$ and $\zt$ is an even coordinate of degree
$2$. From the coordinate descriptions of $f$ we see that there is no obvious way of associating with $f$ an
automorphism of $\R[\zvy_1, \zvy_2, \zt]$. As homogeneity  bundles are obtained by gluing trivial parts of the
form $U\times \R^\dd$ by means of automorphisms of the standard homogeneity  space $\R^\dd$, there is no
obvious way of gluing parts of the form $U\times \Pi\R^\dd$, where $\Pi\R^\dd$ is the $N$-manifold defined by the polynomial super-algebra $\R[\zt_1, \ldots, \zt_n]$
in which the super-degree of $\zt_j$ coincides with the degree $w_i$ of the coordinate $y_j$ on $\R^\dd$.
Therefore, in general, there is no canonical nontrivial procedure for constructing a super-bundle from a given homogeneity  bundle. Note, however, that in some special
cases such a procedure does exist. For example, from  a given  double vector bundle we may obtain a $N$-manifold of
degree $2$ in a canonical way (\cite{GR}).
%%have obvious `superization' procedure $\Pi$ which associates an N-manifold $\Pi M$ of degree $n$ with a
%homogeneous super-bundle $M$ of degree $n$: we assume the parity of local coordinates coinciding with the
%parity of the weight; this agrees with gluing, as the weights are preserved. For vector bundles it reduces to
%the standard superization $E\mapsto \Pi E$.
\end{remark}

\bigskip
Recall that any real vector bundle of rank $n$ has its associated principal $\GL(n,\R)$-bundle, the {\it
frame bundle}, and {\it vice versa}: with any principal $\GL(n,\R)$-bundle a vector bundle of rank $n$ is
canonically associated. A similar procedure can be performed for graded bundles.

First, for a homogeneity  space $V$ of rank $\dd$ let us consider the space $F(V) = \text{Iso}(\R^\dd, V)$ of
isomorphisms $\zf:\R^\dd\ra V$ of graded spaces, i.e., diffeomorphisms $\zf:\R^N\ra V$ respecting the
actions of the semigroup $(\R_{\ge 0}, \cdot)$, where $N  = \sum_{i=1}^n d_i$, as usual. A choice of global
homogeneous coordinates on $V$ gives an identification of $F(V)$ with  $\GH(\dd, \R)$, which is a Lie group.
More generally, for a graded bundle $M$ of rank $\dd$ fibred over $M_0$, $\zp:M\to M_0$, the disjoint
union
$$
F(M) = \coprod_{p\in M_0} F(\zp^{-1}(p))
$$
has a natural smooth manifold structure and a smooth right action $F(M)\times \GH(\R^\dd)\to F(M)$ which turns
$F(M)$ into a principal $\GH(\R^\dd)$-bundle. We shall call $F(M)$ the {\it associated frame bundle}. One can
reconstruct the homogeneity  bundle structure from $F(M)$ in the standard way: given an arbitrary principal
$\GH(\R^\dd)$-bundle $(P, \zp, M_0)$ and a left $\GH(\R^\dd)$-action on a manifold $Q$ one defines a fiber
bundle $(P_Q, \zp_Q, M_0)$ with the typical fiber $Q$, where $P_Q = P\times_{\GH(\R^\dd)} Q$. By taking
$Q=\R^\dd$ and $P = F(M)$ we recover $M \simeq P_Q$.

\bigskip
By convention, a graded bundle $M$ of degree $n$ is also a graded bundle of degree $m$, for $m>n$. Let
$\As^k$ be the subalgebra of $\As=\As(M)$ generated by functions from $\As$ of degree $\le k$. We have the following
sequence of associative algebra inclusions
$$
\Csi(M_0) = \As^0 \hookrightarrow \As^1 \hookrightarrow \As^2 \hookrightarrow \ldots \hookrightarrow \As^n =
\As
$$
which gives rise to the sequence
\be{affseqn}
M_0 \leftarrow M_1 \leftarrow M_2 \leftarrow \ldots \leftarrow  M_n = M
\ee
of affine bundle projections.
%KOREKTA Here, $M_k$ is a graded subbundle of $M$ of degree $k$ defined by vanishing
%of all local coordinates of degrees $>k$.

Note also that, like for graded spaces, any  graded bundle of degree $n$ is canonically equipped
with an action $h:\Real_{\ge 0}\times M \to M$ of the multiplicative semigroup $(\Real_{\ge 0},\cdot)$ defined
locally in unified homogeneous coordinates by
\be{homact}
h_t(y^1, y^2, \ldots, y^{N+m}) = (t^{w_1}\, y^1, \ldots, t^{w_{N+m}}\, y^{N+m})=(t^{w_1}\, y^1, \ldots,
t^{w_{N}}\, y^N, y^{N+1}, \ldots, y^{N+m}),
\ee
where $h_t = h(t, \cdot)$ for $t\ge 0$. As transition functions preserve the degree, this action is globally well
defined. The maps $h_t$ we shall call {\it homotheties} of the graded bundle. The action by homotheties
on the graded bundle is determined by its restriction to the multiplicative one-parameter group of
positive reals and, in turn, by its infinitesimal generator: the {\it weight vector field} $\zD_M$ on $M$. In
homogeneous local coordinates $(y^1,\dots,y^{N+m})$ this vector field reads
\be{wvf}\zD_M=\sum_{j=1}^{N+m} w_j y^j\partial_{y^j}=\sum_{i=1}^N w_i y^i\partial_{y^i}\,.
\ee
Like in the case of graded spaces, the weight vector field is complete and, if $t\mapsto \exp{(t\zD_M)}$
is the flow of diffeomorphisms it generates, we have $\exp{(t\zD_M)}=h_{e^t}$ for all $t\in\R$. Thus, the
action by homotheties is completely determined by the weight vector field and {\it vice versa}. The weight
vector field, exactly like the action $h$ of $(\R_{\ge 0},\cdot)$ by homotheties, completely encodes the
homogeneity structure. Indeed, $h_0=\zp$ is the projection in the graded bundle and $\zD_M$, being
tangent to the fibers, defines on them the structures of graded spaces of rank $\dd$. We can formulate
these observations as follows.
\begin{theo}
A graded bundle $M$ of degree $n$ and the weight vector $w=(w_j)$, $w_j=0,1,\dots,n$, can be equivalently
defined as a manifold equipped with a complete vector field $\zD_M$ -- the weight vector field -- for which
there is an atlas $\{(W,y_W):W\in\cal W\}$ of coordinate charts, with $W$ being invariant with respect to
the flow of diffeomorphisms induced by $\zD_M$, such that $\zD_M$ takes the standard form (\ref{wvf}) in each
system of coordinates $(y^j)=(y^j_W)$, $W\in\cal W$. A smooth map between two graded bundles is a
graded bundle morphism if and only if it relates the corresponding weight vector fields.
\end{theo}

The intrinsic properties of the action $h$ ensuring that we deal with a graded bundle will be described
in the next section.

\begin{example}\label{e:2vb} Let $F$ be a double vector bundle with two vector bundle structures corresponding to homotheties $h_t^a$,
and the Euler vector fields $\zD^a$, $a=1,2$. According to \cite[Theorem 3.2]{GR}, there is an atlas for $F$
with local coordinates which are simultaneously homogeneous with respect to both Euler vector fields.
In other words, we can simultaneously write them in the form
$$\zD^a=\sum_{i=1}^Nw^a_iy^i\pa_{y^i}\,, \ a=1,2\,,$$
where $w^a_i=0,1$. It follows that $F$ carries a canonical structure of a graded bundle of degree 2 with
homotheties $h_t=h^1_t\circ h^2_t=h^2_t\circ h^1_t$ and the weight vector field $\zD_F=\zD^1+\zD^2$. In other
words, the coordinate $y^i$ is of degree 0,1, or 2, if, respectively, $y^i$ is of degree 0 with respect to
both, is of degree 0 with respect to one and degree 1 with respect to the other, and is of degree 1 with
respect to both Euler vector fields. In particular, for each manifold $M_0$, the bundles $TTM_0$ and
$T^*TM_0\simeq TT^*M_0$ are canonically graded bundles of degree 2 over $M_0$.

This picture can be easily generalized to a $n$-vector bundle which is canonically a graded bundle of
degree $n$ and whose weight vector field is the sum of the $n$ (commuting) Euler vector fields representing
$n$ compatible vector bundle structures.
\end{example}

\begin{example} Let $M_0$ be a smooth manifold of dimension $m$. Recall that two curves $\za$, $\zb: \Real\to
M_0$ have the same $r$-jet ($r\geq 0$) at $0\in\Real$, if for any smooth function $x$ on $M_0$, the difference
$x\circ\za - x\circ\zb$ vanishes at $0\in \Real$ up to the order $r$. In such a case we write $\za\sim_r\zb$,
and it is clear that $\sim_r$ is an equivalence relation. The coset of $\za$ with respect to this equivalence
is called the $r$-jet of $\za$ and will be denoted by $\jet{\za}_r$. The space of all $r$-jets will be denoted
by $\jetbndl{r}_0$. It is a bundle over $M_0$ with the projection $\pi(\jet{\za}_r)= \za(0)=[\za]_0$. For
example, $\jetb{1}{M_0} \simeq T M_0 $ is the tangent bundle of $M_0$. The subset of $r$-jets at $\zvy\in M_0$
will be denoted by $\jetspacen{r}{\zvy}{M_0}$. In the literature, $\jetbndl{r}_0$ is  called {\it the $r$th
tangent bundle}. The elements of $\jetbndl{r}_0$ are also called {\it velocities of order } $r$ on $M_0$.

Let $(x_a)$ be local coordinates around $\zvy\in M_0$ and let $\jet{\za}_r\in \jetspacen{r}{\zvy}{M_0}$. Let
us denote the Taylor coefficients of $x_a(\za(\cdot))$ in the following way
$$
x_a(\za(t)) = x_a(\zvy) + t\cdot \dot{x}_a(\jet{\za}_r) + \frac{t^2}{2!}  \ddot{x}_a(\jet{\za}_r) + \ldots +
\frac{t^r}{r!}  x^{(r)}_a(\jet{\za}_r) + o(t^r).
$$
Thus $(x_a, \dot{x}_b, \ddot{x}_c, \ldots)$ form the so called {\it adapted coordinate system} in
$\jetbndl{r}_0$. % and $x_a^{(i)}$ has degree $i$.
We shall write $[\za]_r\sim 0$ and say that $[\za]_r$ {\it vanishes}, if $x_a^{(i)}([\za]_r)=0$ for
$i=1,\dots,r$ and all $a$ -- this does not depend on the choice of local coordinates $(x_a)$. By convention,
$[\za]_0\sim 0$ for all $\za$.

One can prove easily that $\jetbndl{r}_0$ is a  graded bundle of degree $r$ and rank $(m,\dots,m)$. The
canonical action of $(\Real, \cdot)$ on $\jetbndl{r}_0$ is given by
\be{e:actJet}
t. \jet{\za}_r := \jet{t.\za}_r,
\ee
where $(t.\za)(s) = \za(ts)$. In coordinates, $x_a^{(k)}(t. \jet{\za}_r) = t^k \cdot x_a^{(k)}(\jet{\za}_r)$;
hence $x_a^{(k)}$ has degree $k$. The sequence (\ref{affseqn}) reads
%\be{e:qmap}
$$
\xymatrix{ \jetbndl{n}_0 \ar[r]^{q_n} & \ldots  \ar[r]^{q_3} &  \jetbndl{2}_0 \ar[r]^{q_2} & TM_0 \ar[r]^{q_1}
& M_0, }
$$
%\ee
where $q_i(\jet{\za}_i) = \jet{\za}_{i-1}$ for $i=1, 2, \ldots, n$. There are also  canonical inclusions
%\be{e:imap}
$$
\zi_k: TM_0 \hookrightarrow \jetbndl{k}_0,
$$
%\ee
defined by $\zi_k(\jet{\za}_1) = \jet{\tilde{\za}}_k$, where $\tilde{\za}(t) = \za(\frac{1}{k!} t^k)$,
$k\in\N$. We have
\be{e:TMjet}
\zi_k(t^k\cdot v) = t.\zi_k(v),
\ee
for $v\in TM_0$. Indeed, if $v = \jet{\za}_1$, then $\zi_k(t^k\cdot v)$ is the class of the curve $s\mapsto
(t^k.\za)(\frac{1}{k!} s^k) = \za(\frac{1}{k!} s^kt^k)$, and $t.\zi_k(v) = \jet{t.\tilde{\za}}_k$,
$(t.\tilde{\za})(s) = \tilde{\za}(st) = \za(\frac{1}{k!} s^k t^k)$. In coordinates, for $r\geq 1$,
$x_b^{(r)}(\zi_k(\frac{\partial}{\partial x_a}))$ is $1$, when $r=k$ and $a=b$, and zero otherwise. Note
that, if $\jet{\za}_{k-1} \in\jetbndl{k-1}_0$ vanishes, then there is a well-defined iterated differential
\be{e:ziTM}
 v := \left. \frac{d^k}{dt^k}\right| _{t=0} \za  =
 \sum_a x_a^{(k)}(\jet{\za}_k) \, \frac{\partial}{\partial x_a} \in TM_0,
\ee
hence $\zi_k(v) = \jet{\za}_k$.
\end{example}

\begin{theo}
Any smooth map $\Phi: M_0\ra N_0$ between manifolds $M_0$ and $N_0$ induces a map $T^r\Phi:T^rM_0\ra T^rN_0$,
defined by $T^r\Phi([\za]_r)=[\Phi\circ\za]_r$, which is a graded bundle morphism.
\end{theo}
\bepf
The proof consists of direct calculations using the rule of differentiating compositions and products of
functions. To be instructive, let us start with the case $r=2$. In local coordinates $(x_a')$ in $N_0$ and
$(x_b)$ in $M_0$, we get
\bea%\label{coord}
x_a'=\Phi_a(x),\nn \\
\dot{x}_a' = \sum_b \frac{\partial \Phi_a}{\partial x_b}(x) \, \dot{x}_b,\nn \\
\ddot{x}_a' = \sum_b \frac{\partial \Phi_a}{\partial x_b}(x) \,\ddot{x}_b + \frac{1}{2} \sum_{b,c}
\frac{\partial^2 \Phi_a}{\partial x_b\partial x_c}(x)\, \dot{x}_b \dot{x}_c\,.\nn
\eea
It is easy to see that the above transformations preserve the degree. The general case follows directly from
the identity
$$(T^r\Phi)(t.[\za]_r)=t.T^r\Phi([\za]_r)\,.$$

\epf

\noindent More information about the $r$th tangent bundle $T^rM_0$ and general jet bundles can be found, for
instance, in \cite{KMS,dLR,Sau,YI}.

\section{Homogeneity structures}
The graded bundles have been defined in the preceding paragraph as glued from the standard ones of the
form $U\ti\R^\dd$. Here, we will give an intrinsic characterization of graded bundles which is easily
verifiable and useful for applications.

Let $h:[0,+\infty)\ti M\ra M$ be a smooth action of the multiplicative semigroup $(\Real_{\geq0}, \cdot)$ of non-negative
reals on a manifold $M$. `Smooth' means that the map $h$ can be extended to a smooth map on
$(-\ze,+\infty)\ti M$ for some $\ze>0$. We shall also speak about {\it local actions}. They are smooth maps
$h:[0,\ze)\ti M\ra M$ satisfying $h_t\circ h_s=h_{ts}$ whenever $h_t=h(t,\cdot)$, $h_s=h(s,\cdot)$, and $h_{ts}$ make sense, i.e., $0\le t,s,ts<\ze$.

It is explained in \cite{GR} that, under an additional condition, $h$ determines a vector bundle structure on
$M$ for which $h$ coincides with the action by homotheties in this vector bundle. We are going to generalize
this result and give a sufficient and necessary condition for $h$ ensuring that there exist a structure of a
graded bundle of degree $n$ on $M$ whose homotheties coincide with $h_t$, $t\ge 0$.

\begin{definition} A {\it homogeneity structure} of degree $n$  on a smooth manifold $M$ is a
smooth action
$$
h:[0, \infty) \times M \to M
$$
of the multiplicative semigroup $(\Real_{\geq 0}, \cdot)$ such that, for $\zvf_n: M \to \jetbndln$ defined by
\be{e:defzvf}
\zvf_n(p) = \jet{\za_p}_n,
\ee
with $\za_p(t) = h(t, p)$, the following condition is satisfied:
\be{cond}
\zvf_n(p)\ \text{vanishes if and only if }\ p\in M_0 := h_0(M)\,.
\ee
\end{definition}

\medskip\noindent
When $n$ is fixed, we shall also write simply $\zvf$ for $\zvf_n$. We shall also use the notation $ t.p$
for the action by homotheties in $T^nM$. Note that the only difference between
the homogeneous structures considered in \cite{GR} and the homogeneity structures of degree $n$ here is in the
order of derivatives of $\za_p$ that can vanish only on $M_0$. This order is here an arbitrary $n=1,2,\dots,$
instead of just 1.

We do not assume that we deal with an action of the multiplicative $\R$, as a natural extension to $\R$ will follow automatically.
We do not assume also here that $h(1,p)=p\ $ for $p\in M$, so we use the
structure of the semigroup rather than the monoid structure on $(\Real_{\geq0}, \cdot)$. However, it turns
out that $h_1=id_M$ automatically, for any $n$-homogeneity structure.
\begin{prop} If $h:[0, \infty) \times M \to M$ is a smooth action of the multiplicative semigroup $(\Real_{\geq 0}, \cdot)$, then $h_1$ is a smooth submersion onto the submanifold $M_1=h_1(M)$, and $h$ reduced to $M_1$ is a smooth action of the multiplicative monoid $(\Real_{\geq0}, \cdot)$, i.e. $(h_1)_{|M_1}=id$. In particular, any homogeneity structure of degree $n$ on $M$ is actually the monoid action, $h_1=id$.
\end{prop}
\begin{proof} Since $h_1$ is a projection onto $M_1=h_1(M)$, $h_1\circ h_1=h_1$, then $M_1$ is a smooth manifold and $h_1$ is a submersion onto $M_1$ according to \cite[Theorem~1.13]{KMS}. Since $h_t\circ h_1=h_1\circ h_t=h_t$, the action restricts to $M_1$ on which $h_1$ is the identity. The nondegeneracy assumption (\ref{cond}) ensures that $M_1=M$.
\end{proof}

Thus, the nondegeneracy condition implies that we deal with a monoid action. We will see that the converse is also true. It is based on the following observation.
%begin{KOREKTA}
\begin{prop}\label{p:flat}
Let $h: [0, \infty) \times M \to M$  be an arbitrary smooth action
of the semigroup $(\Real_{\geq 0}, \cdot)$ and assume that, for some point $p$,
the curve $\za_p(t) = h_t(p)$ is not constant.
Then, there exists an integer $n$ such that the $n$-th jet of $\za_p$ at $t=0$ does not vanish.
\end{prop}
\begin{proof}
It is easy to see that the image $\za_p([0, \infty))$ is a submanifold with boundary of $M$,  diffeomorphic to $[0, \infty)$, so
the general case reduces essentially  to the case $M=\R$. In fact, we can work with the subset $[0, \infty)\subset\R$. Let us assume that $0\neq p_0 \in [0, \infty)$
and $f(t)=h_t(p_0)$ is flat at $0$. Let as consider the sets $A_n = \{t>0: f(t)\geq t^n\}$, $n\in \N$. There exists $n_0$ such that,
for $n\geq n_0$, $A_n$ is not empty. Since $f$ is flat, $A_n$ is bounded from below by a positive number, and is closed. Therefore
$A_n$ has a minimal element, say $t_n$. Hence, $f(t_n)= t_n^n$ and $f(t)<t^n$ for $t\in (0, t_n)$. Of course, $\lim_{n\to \infty} t_n = 0$.
The partial derivative
of $h(t, x)$ with respect to $x$ at the point $(2, 0)$ is $\frac{\partial h}{\partial x} (2, 0) = \lim_{x\to 0 } \frac{h_2(x)-h_2(0)}{x}$.
Substituting $x:= h(t_n/2, p_0)$, we get
$$\frac{h_2(h_{t_n/2}(p_0))-h_2(0)}{h_{t_n/2}(p_0)} = \frac{f(t_n)}{f(t_n/2)}
> \frac{t_n^n}{(t_n/2)^n} = 2^n\to + \infty\,;$$
a contradiction.
\end{proof}
%end{KOREKTA}

\medskip
An important observation used in the sequel is that, for any smooth action $h$ of the semigroup $(\Real_{\geq0}, \cdot)$, we have
\be{e:l1}
t.\zvf_r(p) = \zvf_r(h_t(p)).
\ee
Indeed, $t.\zvf_r(p) = t. \jet{\za_p}_r = \jet{t.\za_p}_r$, and $(t.\za_p)(s) := \za_p(ts) = h(ts,p) =
h_s(h_t(p))$, so $\jet{t.\za_p}_r = \zvf_r(h_t(p))$, and (\ref{e:l1}) follows.

\begin{lem}\label{L} Let $h:[0,\ze)\times M \to M$, $\ze>0$, be a local smooth action of the semigroup
$(\Real_{\geq0}, \cdot)$ such that $h_0(M) = \{\zvy \}$ is a single point. Define $\za_p(t) = h(t,p)$, and let
$r=1,2,\dots$. Let us assume that  $\jet{\za_p}_{r-1}\sim 0$ for any $p\in M$, so that the map $\zc_r: M\to
T_\zvy M$,
\be{e:defzc}
\zc_r(p) := \kdiff{r} h(t, p)\,,
\ee
is well defined. Then,
\be{e:psi1}
\zc_r(h_t(p)) = t^r\cdot \zc_r(p)
\ee
for all $p\in M$, and the differential $Q_r=\differ_\zvy\zc_r:T_\zvy M\ra T_\zvy M$ is a projection onto a
subspace $E\subset T_\zvy M$ containing the image of $\psi_r$:
\be{e:psi2}
(\differ_\zvy\zc_r)(\zc_r(p)) = \zc_r(p).
\ee
If, in addition, $\zc_r(p) = 0$ only for $p=\zvy$, then $\differ_\zvy\zc_r$ is the identity map on $T_\zvy M$.
\end{lem}

\bepf Let  $\zvf_r: M \to \jetb{r}{M}$ be defined as in
(\ref{e:defzvf}). From (\ref{e:ziTM}) we derive
\be{e:zizc}
\zi_r\circ \zc_r  =\zvf_r.
\ee
Hence,  $\zi_r(\zc_r(h_t(p))) = t.\zi_r(\zc_r(p)) = \zi_r(t^r\cdot \zc_r(p))$, by  (\ref{e:TMjet}) and
(\ref{e:l1}), that justifies (\ref{e:psi1}). Let us consider the maps $H(t), Q: T_\zvy M \to T_\zvy M$ defined
by
\be{e:Htt}
H(t)= H_t := \differ_\zvy h_t,
\ee
\be{e:Qr}
Q_r :=\differ_\zvy\zc_r.
\ee
Like in the proof of \cite[Theorem~2.1]{GR} we find that
\be{e:l2}
H(ts) = H(t)\circ H(s) = H(s)\circ H(t)
\ee
by differentiating $h_{ts} =h_t \circ h_s$ at $\zvy\in M$. In view of the Schwarz's theorem on mixed partial
derivatives,
\be{e:l3}
Q_r = \kdiff{r} H(t).
\ee
By a simple calculation we find that $ \kdiff{r}$ applied to (\ref{e:psi1}) gives (\ref{e:psi2}) (as $\left.
\frac{d^j}{dt^j}\right|_{t=0} h(t, p)= 0$, for $1\leq j\leq r-1$). Applying $\kdiff{r}$ to (\ref{e:l2}), we
get in turn
$$
s^r\cdot Q_r = Q_r\circ H(s)  = H(s)\circ Q_r.
$$
Eventually, by applying $\frac{1}{r!} \left. \frac{d^r}{ds^r}\right|_{s=0}$ to the last expression, we end up
with
$$
Q_r\circ Q_r =  Q_r,
$$
hence $Q_r$ is a projection onto a subspace, say $E$, of $T_\zvy M$. We are left with proving the last
statement of the lemma. By (\ref{e:psi2}) we get $\zc_r(p)\in E$, hence $\zc_r$ is a map from $M$ to $E$ of
maximal rank at $\zvy$. Applying now the Implicit Function Theorem we find that $\zc_r^{-1}(0)$ is a smooth
submanifold around $\zvy$ of dimension $\text{dim }\,T_\zvy M - \text{dim}\,E$. But, according to the
hypothesis, $\zc_r(p) = 0 $ if and only if $p = \zvy$.  Hence, $E = T_\zvy M$ and $Q_r = \differ_\zvy\zc_r$ is
the identity, that completes our proof.

\epf

\begin{theo}\label{thm:main} Any homogeneity structure $h$ of degree $n$ on a manifold $M$
determines a unique structure of a  graded bundle $\zp: M\to M_0$ of degree $n$ for which $h$ coincides
with the canonical action %(\ref{homact})
of the semigroup $(\Real_{\geq0}, \cdot)$ by homotheties on this bundle. In other words, for a homogeneity
structure we can always find an atlas with homogeneous coordinates. The map (\ref{e:defzvf})
establishes an isomorphism of this graded bundle with a graded subbundle of $T^n M$.
\end{theo}
\bepf
Let $\zvf = \zvf_n: M\to \jetbndln$ be given as in (\ref{e:defzvf}). We shall prove that $\zvf$ is a
diffeomorphism onto its image $\zvf(M)$ which is equipped with a homogeneity structure of degree $n$ inherited
from $\jetbndln$. Moreover, $\zvf$ intertwines the homogeneity structure $h$ with that on $\jetbndln$, as
we have seen in (\ref{e:l1}). Because $h_0: M\to M$ is a smooth projection onto $M_0 := h_0(M)$ (as $h_0\circ
h_0 = h_0$), $M_0$ is a submanifold of $M$ (\cite{KMS}, Theorem~1.13). Moreover, for $\zvy \in M_0$ there is a
neighbourhood $U\subset M$ of $\zvy$ and and there are local coordinates $(q_a, x_i)$ on $U$ in which $h_0$ has the form
\be{e:locsubm}
h_0(q_a, x_i) = (q_a, 0)\,,
\ee
i.e., $h_0$ is a local submersion along $M_0\subset M$.

The strategy depends on proving that there is a graded subbundle $\widetilde{M}$ of $\jetbndln$ such
that, for each $\zvy\in M_0$, the map $\zvf$ is a diffeomorphism from a neighbourhood $U$ of $\zvy$ onto an
open subset of $\widetilde{M}$. Then, using the property (\ref{e:l1}), we shall show that $\zvf$ is a
diffeomorphism onto the whole $\widetilde{M}$.

We want to show first that $\differ_\zvy\zvf$ is one-to-one for $\zvy\in M_0$. Because $\zvf$ fixes points of
$M_0$ (we shall consider $M$ as a submanifold of $\jetbndln$) and the fibers $M_\zvy := h_0^{-1}(\zvy)$ are
transversal to $M_0$ at $\zvy$ (by (\ref{e:locsubm})), it is enough to prove that
$\differ_\zvy\widetilde{\zvf}$ is injective, where $\widetilde{\zvf} = \zvf_{|M_\zvy}$. In other words, it is
enough to consider the case where $M_0 = \{\zvy\}$ is a single point. From now on, we assume that $M=M_\zvy$ and
$\widetilde{\zvf} = \zvf$.

We are going to define inductively a series of maps  $\zc_r: N_r\to V_r:=T_\zvy N_r$, $r=1, 2, \ldots, n$, for
some decreasing sequence of submanifolds $N_1 := M \supseteq N_2 \supseteq \ldots \supseteq N_n$. The formula for $\zc_r$ is the same as in (\ref{e:defzc}), but in order to have it correctly defined, the submanifold $N_r$ will be chosen in such
a way that $\zc_{r-1}(p) = 0 $ for $p\in N_r$. We are going to prove, inductively with respect to $r$, that
$\zc_r^{-1}(0)$ is a submanifold  around $\zvy$. It will serve as the next submanifold $N_{r+1}$.

Let $r=1$, $V_1 := T_\zvy M$. In view of  Lemma~\ref{L} we find that $Q_1\circ Q_1 = Q_1$ and $Q_1(\zc_1(p))=
\zc_1(p)$, where $Q_1 = \differ_\zvy\zc_1$. Hence, $Q_1$ is a projection on a subspace $E_1\subseteq V_1$ and
$V_1 = E_1 \oplus K_1$, where $K_1 = \text{ker}\,Q_1$. Moreover, $\zc_1(N_1)\subseteq E_1$, hence $\zc_1$ is a
map of maximal rank at $\zvy$. It follows from the Implicit Function Theorem that $N_2:= \zc_1^{-1}(0)\cap
U_1$ is a smooth submanifold of $M$ for some open neighbourhood $U_1\subset M$ of $\zvy$ and, moreover,
$T_\zvy N_2 = K_1$.

Let us assume that we have already defined submanifolds $N_1, \ldots, N_r$, and $N_r$ is of the form
$\zc_{r-1}^{-1}(0)\cap U_{r-1}$ for an open neighbourhood $U_{r-1}\subset N_{r-1}$ of $\zvy$. It follows that
$\zc_r: N_r\to T_\zvy N_r =:V_r$ is well defined. From $\zc_{r-1}(h_t(p)) = t^{r-1}\cdot \zc_{r-1}(p)$ (see
(\ref{e:psi1})), we find that $h(t, p) \in \zc_{r-1}^{-1}(0)$ for $p\in N_r$ and  any $t\geq 0$. Hence, $h$
gives rise to a local action on $N_r$. According to Lemma \ref{L}, $\zc_r$ is a map from $N_r$ to a subspace
$E_r\subseteq V_r$ of the maximal rank at $\zvy$, and $Q_r = \differ_\zvy \zc_r \in \text{End}(V_r)$ is the
projection on $E_r$. Hence, $\zc_r^{-1}(0)$ is a submanifold around $\zvy$, what enables us to define a
submanifold $N_{r+1}$, such that $V_{r+1} = T_\zvy N_{r+1}$ is the kernel $K_r$ of $Q_r$. We have $V_r  = E_r
\oplus K_r$ and $V_{r+1} = K_r$. The inductive construction of $\zc_1, \ldots, \zc_n$ is completed.

By the hypothesis of our theorem, $\zc_n^{-1}(0) = \{\zvy\}$, hence $V_n =E_n$ and $K_n = \{0\}$. In view of
Lemma~\ref{L}, $\differ_\zvy \zc_n$ is the identity on $E_n$. We have proved that  the tangent space to $M$ at
$\zvy$ has a canonical decomposition
\be{e:TMdec}
T_\zvy M  = E_1\oplus \ldots \oplus E_n
\ee
where $E_r$ is the image of the projection $Q_r\in\text{End}(V_r)$ and $V_r = E_r\oplus E_{r+1}\oplus
\ldots\oplus E_n$ for $r=1,2 , \ldots, n$.  Let $q_r^n: \jetbndl{n}\to \jetbndl{r}_0$ be the canonical
projection. We have $q_r^n\circ \zvf_n = \zvf_r  =\zi_r\circ \zc_r$ on $N_r$ (by (\ref{e:zizc})), so
\be{e:zizvf}
\differ_\zvy q_r^n\circ \differ_\zvy\zvf_n = \differ_\zvy \zvf_r  =\differ_\zvy\zi_r\circ \differ_\zvy\zc_r
=\differ_\zvy\zi_r\circ Q_r
\ee
on $V_r$. Let us assume for a moment that $e=(e_1, \ldots, e_n)$, $e_i\in E_i$, is in the kernel of
$\differ_\zvy\zvf_n$. Applying (\ref{e:zizvf}) to $e$, with $r=1$, we find that $0 = Q_1(e) = e_1$, so $e \in
V_2$ and we may apply again (\ref{e:zizvf}) to $e$, but now with $r=2$. Since $\differ_\zvy\zi_r$ is
injective, $0 = Q_2(e) = e_2$. Continuing in the same manner with $r= 3, \ldots$, we get $e_r=0$ for any $r$.
Hence $\differ_\zvy\zvf_n$ is injective, as we have claimed.

It follows that there exists a neighbourhood $U\subset M$ of $\zvy$ such that $\zvf_{|U}$ is a diffeomorphism
onto its image $\widetilde{U}$, $\widetilde{U}\subset \jetspacen{n}{\zvy}{M}$. It is possible to find local
coordinates $(x_{r, a})$, $1\leq r\leq n$, $1\leq a\leq \text{dim}\,E_r$, on $U$ (if necessary, we replace $U$
with a smaller neighbourhood of $\zvy$) such that $\frac{\partial}{\partial x_{r, a}}(\zvy)\in E_r$ for any
$1\leq r\leq n$ and $1\leq a\leq\text{dim}\,E_r$. Let $(x_{r, a}^{(s)})$, $r,s\geq1$,  be the adapted
coordinate system on $\jetspacen{n}{\zvy}{M}$. Let us see that the restrictions of $x_{r, a}^{(r)}$ ($r, a$ as
above) to some neighbourhood of $\zvy$ in $\widetilde{U}$ form a local coordinate system. Obviously, the
vectors $X_{r, a} := (\differ_\zvy\zvf)(\frac{\partial}{\partial x_{r,a}}(\zvy))$ span the tangent space to
$\widetilde{U}$ at $\zvy$. Moreover, $\differ_\zvy\zc_r$ fixes $\frac{\partial}{\partial x_{r, a}}(\zvy)$ and,
from the coordinate description of $\zi_r$ given below (\ref{e:TMjet}),
$(\differ_\zvy\zi_r)(\frac{\partial}{\partial x_{r, a}}(\zvy)) = \frac{\partial}{\partial x_{r,
a}^{(r)}}(\zvy)$. Hence, using (\ref{e:zizvf}), we find that $\differ_\zvy q^n_r$ applied to $X_{r,a}$ gives
$\frac{\partial}{\partial x_{r, a}^{(r)}}(\zvy) \in \jetspacen{r}{\zvy}{M}$. Therefore
\be{e:tngvctrs}
X_{r,a} = \frac{\partial}{\partial x_{r, a}^{(r)}}(\zvy) + \text{higher order terms},
\ee
where the `higher order terms' means a linear combination of vectors $\frac{\partial}{\partial x_{u,
b}^{(s)}}(\zvy)$ with $s>r$, so the differentials $\xd x_{r, a}^{(r)}$ are linearly independent on
$T_\zvy\widetilde{U}$.
Without loss of generality we may assume that $\widetilde{U} = \zvf(U)$.
Hence, the chart
$(\widetilde{U}, (x_{r, a}^{(r)}))$ provides an identification of $\widetilde{U}\subset
\jetspacen{n}{\zvy}{M}$ with an open neighbourhood, say  $I$, of $0$ in $\Real^N$. Let us define
\be{e:extn}
\widetilde{M} = \bigcup_{t\geq 1} t.\widetilde{U}.
\ee
Thanks to the $r$-homogeneity of $x_{r, a}^{(r)}$, the chart $(t.\widetilde{U}, (x_{r, a}^{(r)}))$ gives a
diffeomorphism with another open neighbourhood, say $t.I$,  of $0$ in $\Real^N$. This notation is consistent
with the action of the semigroup $(\Real_{\ge 0}, \cdot)$ on $\Real^N$ described in Lemma \ref{L:e}. As
$\bigcup_{t\geq 1} t.I = \Real^N$, we get a diffeomorphism of $\widetilde{M}$ with $\Real^N$. Hence
$\widetilde{M}$ is a submanifold of $\jetspacen{n}{\zvy}{M}$, invariant under the canonical action of the
semigroup $(\Real_{\ge 0}, \cdot)$. Moreover, for $t>0$, the map $p\mapsto t^{-1}.\zvf(h_t(p))$ is a
diffeomorphism of $h_t^{-1}(U)$ onto $t^{-1}.\widetilde{U}$. But, according to (\ref{e:l1}),
$t^{-1}.\zvf(h_t(p)) = \zvf(p)$. Hence, $\zvf$ is a diffeomorphism of $M$ onto $\widetilde{M}$ intertwining
the homogeneity structures on $M$ and $\jetspacen{n}{\zvy}{M}$.

In the general case, when $M_0$ is not necessarily a single point, the arguments for the last sentence are very
similar. The map $\zvf$ is a local diffeomorphism along $M_0$, hence a diffeomorphism from an open
neighbourhood $U\subset M$ of $M_0$ onto a submanifold $\widetilde{U}\subset\jetbndln$ and $\widetilde{U}$
contains the image of the zero section of $\jetbndln$. Let $\widetilde{M}$ be defined  as in (\ref {e:extn}).
Since $t^{-1}.\zvf(h_t(p)) = \zvf(p)$, working separately in each fiber of $M\to M_0$ as above, we find that
actually $\zvf$ identifies $M$ with $\widetilde{M}$.  Let us pull back the structure on $\widetilde{M}$ to $M$
by means of $\zvf$. In this way we equipped $M$ with a structure of a graded bundle of degree $n$
compatible with the given homogeneity structure. This completes the proof of the existence part of the
theorem.

According to our definition of a homogeneity  bundle, the uniqueness part of the theorem is trivial. \epf

%KOREKTA
\begin{remark} The reason why we have considered the action $h$ of non-negative reals and not just all reals is that this version of Theorem~\ref{thm:main} is more general, as a natural extension to an action of $(\R,\cdot)$ follows for free.
Indeed, let us notice that since we can choose coordinates homogeneous,
$$
x_{r,a}(h_t(x)) = t^r \, x_{r, a}, \,\,t\geq 0,
$$
we can canonically extend the $\Real_{\geq0}$-action to $(\Real, \cdot)$ action by simply allowing $t$
on the r.h.s. to take negative values. Moreover, this extension is unique.
It is enough to show that the action by $-1$, which we denote by $f: M\to M$, is uniquely determined by
the homotheties $h_t$, $t\geq 0$. As $f(h_t(p))= h_t(f(p)) = h_{-t}(p)$ for $t\geq 0$, $f$ is
a homogeneous bundle morphism. Hence, in some homogeneous coordinates $(x_{r, a})$ on $M$,
we have $h_t^*(x_{r,a}) = (-t)^r f^*(x_{r, a})$ for $t<0$. It follows that,
$$
\frac{1}{r!}\left. \frac{d^{r}}{dt^{r}}\right| _{t=0^+} h_t^*(x_{r,a})  = x_{r,a}, \quad
\frac{1}{r!}\left. \frac{d^{r}}{dt^{r}}\right| _{t=0^-} h_t^*(x_{r,a})  = (-1)^r\cdot f^*(x_{r,a}).
$$
Therefore, by the smoothness of $h$ at $t=0$, $f^*(x_{r, a}) = (-1)^r x_{r, a}$ and $h_t^*(x_{r,a}) = t^r \cdot x_{r,a}$
for any $t\in \Real$, that finishes our proof.
\end{remark}

Actually, the assumptions of the above theorem above can be reformulated: we can replace the nondegeneracy condition with the condition that we deal with a monoid action, i.e. $h_1=id_M$. For the semigroup action the nondegeneracy is not automatic even when some jets are nontrivial, as the example $h:[0,\infty)\ti\R^2\ra\R^2$, $h_t(x,y)=(tx,0)$ shows.

Let $h:[0, \infty) \times M \to M$ be an arbitrary smooth action of the monoid $(\Real_{\geq 0}, \cdot)$ and assume that there exist a point $p\in M$ and $t>0$ such that $h_t(p)\neq h_0(p)$. According to Proposition~\ref{p:flat}, there exists an integer $n$ such that the $n$-th jet of $h_t(p)$ does not vanish at $t=0$. We claim that there always exists a common integer $n$ such that the $n$-th jet of $t\mapsto h_t(p)$ at $t=0$ does not vanish for any $p\in M\setminus M_0$. Taking the lowest such $n$, we obtain  a homogeneity structure of degree $n$ on $M$ defined by $h$.

To prove this fact it is enough to modify slightly the proof of Theorem~\ref{thm:main}. First, by working in the fibers of the projection $h_0$,
the inductive procedure of Theorem~\ref{thm:main}  gives a sequence of submanifolds
$M_\zvy= N_1\supseteq N_2 \supseteq \ldots$ such that $N_{r+1} = \zc_r^{-1}(0)$ near $\zvy$, where $\zc_r: N_r\to T_\zvy N_r$
as in (\ref{e:defzc}) and $M_\zvy = h_0^{-1}(\zvy)$ near $\zvy$. Thanks to Proposition~\ref{p:flat},
it is impossible that $\zc_r, \zc_{r+1}, \zc_{r+2}, \ldots $
is a sequence of the zero maps. Hence, the dimension argument shows that actually, for some $n$, we have  $N_{n+1} = \{\zvy\}  =\zc_n^{-1}(0)$ and we obtain  a finite decomposition,
$$
T_\zvy M  = E_1^\zvy\oplus \ldots \oplus E_n^\zvy,
$$
as in Theorem~\ref{thm:main}. Let us define
$$
H_t^\zvy: T_\zvy M \to T_\zvy M, \quad H_t^\zvy :=  \differ_\zvy h_t,
$$
as in~(\ref{e:Htt}). According to the proof of Theorem~\ref{thm:main}, there exist homogeneous coordinates $(x_{r,a}')$ on $M$,
$1\leq r\leq n$, $1\leq a\leq\text{dim}\, E_r^\zvy$,
in which the action $h$ reads as $h_t^*(x_{r,a}')  = t^r\cdot x_{r,a}'$, namely, $x_{r, a}' = x_{r, a}^{(r)}\circ \zvf$.
Therefore $E_j^\zvy$, $1\leq j\leq n(\zvy)$, is
an eigenspace of the linear map $H_t^\zvy$ with the eigenvalue $t^j$. In the general case, when $M_0$ is connected but is not necessarily a single
point, thanks to the continuity of the map $t\mapsto H_t^\zvy$ the dimension of $E_j^\zvy$ cannot jump, hence $n=n(\zvy)$
is constant. By the restriction of some homogeneous coordinates on $T^n M$ we obtain homogeneous coordinates for $M$.

Summarizing, we can write the following.
\begin{theo} A smooth action $h:[0,\infty)\ti M\ra M$ of the multiplicative semigroup $(\R_{\ge 0},\cdot)$ defines  on a connected manifold $M$ a homogeneity structure of some degree $n$ (and thus a graded bundle structure) if and only if it is a monoid action, $h_1=id_M$.
In this case the action can be uniquely extended to an action of the monoid $(\R,\cdot)$ of multiplicative reals. In particular, the category of graded bundles with their morphisms is equivalent to the category of $(\R,\cdot)$-manifolds and equivariant maps.
\end{theo}
\begin{cor} A submanifold $X$ of a graded bundle $M$ is a graded subbundle of $M$ if and only if it is invariant with respect to the homotheties associated with the graded bundle structure on $M$.
\end{cor}
\bepf Let $X$ be a submanifold of a graded bundle $M$ over $M_0$ of degree $n$ and assume that $X$ is
invariant with respect to homotheties $h_t$ of $M$. It is easy to see that the $(\Real_{\geq0}, \cdot)$-action
$h$ by homotheties, reduced to $X$, is a homogeneity structure on $X$. This implies that there is a unique
graded bundle structure on $X$ over the submanifold $X_0=h_0(X)\subset M_0$ for which $(h_t)_{\mid X}$
are homotheties. This graded bundle structure is isomorphic to the graded subbundle
$\phi(X)\subset\phi(M)\subset T^nM$, thus canonically a subbundle of $M\simeq\phi(M)$.

 \epf

\section{Double homogeneity structures}
By analogy with our definition of a double vector bundle \cite{GR} we propose the following.
\begin{definition}  A {\it double homogeneity structure} or {\it double graded bundle} (DHS, for short) of degree $(m, n)$ on a manifold $F$ consists of two homogeneity structures $h^1$, $h^2$ of degrees $m$ and $n$, respectively, such that $h^1_t\circ
h^2_u = h^2_u\circ h^1_t$ for $u,t\in\Real_{\ge 0}$. More generally, an {\it $r$-tuple homogeneity structure} or {\it $r$-tuple graded bundle} of
degree $(m_1,\dots,m_r)$ on $F$ consists of $r$ pairwise commuting--in the above sense--homogeneity structures
on $F$. In particular, a {\it double graded bundle} is a manifold equipped with two graded bundle structures which are {\it compatible} in the sense that the corresponding homotheties commute, i.e., the corresponding
homogeneity structures form a double homogeneity structure.
\end{definition}

\noindent A DHS on $F$ gives rise to the following diagram of four graded bundles
\[
\xymatrix{
F \ar[r]^{{h}^2_0} \ar[d]_{h^1_0} & F_2 \ar[d]^{\underline{h}^2_0}\\
F_1 \ar[r]_{\underline{h}^1_0} &  M }
\]
where $F_i = h^i_0(F)$, $i=1,2$, $M= \pi(F)$ for $\pi:= h_0^1\circ h_0^2$, and the homogeneity structures on
$F_i$ are obtained by the restrictions $\underline{h}^i$ of the corresponding actions by homotheties on $F$. According to
Theorem~\ref{thm:main}, each of the four homogeneity structures in the diagram comes from a unique graded bundle structure of degree $m$ or $n$. A function $f\in\Csi(F)$ is called {\it homogeneous of (bi-)degree}
$(r, s)\in\N\times \N$ if
$$
f\circ h_t^1 =t^r \cdot f\ \text{and}\ f\circ h_t^2 =t^s \cdot f.
$$

\begin{example} Let $M$ be a smooth manifold, and $m,n\in\N$. Let us consider the following equivalence
relation on the set of smooth maps $\za:\R^2\ra M$: $\za \sim \za'$ if and only if for each smooth map
$x:M\to\Real$ the partial derivatives $\frac{\partial^{r+s}}{\partial^r t\partial^s u}$ at the point $(0, 0)$
of $x\circ \za$ and $x\circ \za'$ coincide for any $0\leq r\leq m$ and $0\leq s\leq n$. In particular,
$\za(0,0) = \za'(0, 0)$. Let us define the jet bundle $\jetbndlmn{F}$ as the collection of all cosets of the
relation $\sim$. A coset of $\za$ will be denoted with $[\za]_{m,n} \in \jetbndlmn{F}$. The bundle projection
$q: \jetb{m, n}{M} \to M$ reads $q([\za]_{m,n})= \za(0,0)$. The bundle $\jetb{m, n}{M}$ is equipped with the
canonical DHS induced by the following two actions of $(\R_{\ge 0},\cdot)$ on a representative
$\za:\Real\times\Real \to M$ of $[\za]_{m,n}$:
$$
(h^1_s \za)(t, u)  = \za(st,u),\quad (h^2_s \za)(t, u) = \za(t, s u).
$$
We shall usually write $t._i x$ for $h^i(t, x)$, $i=1,2$. Let us observe that we have a canonical isomorphism
\be{e:jetiso}
I: \jetb{m, n}{M} \to \jetb{n}{\jetb{m}{M}}, \quad [\za]_{m, n}\mapsto [\zb]_n\,,
\ee
where $\zb:\Real\to \jetb{m}M$ is given by $\zb(u) = [t\mapsto \za(t, u)]_m$.

Similarly, one can define the isomorphism $I': \jetb{m, n}{M} \to \jetb{m}{\jetb{n}{M}}$ such that $I'\circ
I^{-1}$ gives a canonical isomorphism $\jetb{n}{\jetb{m}{M}}\to\jetb{m}{\jetb{n}{M}}$.
For $m=n=1$ we recover the well-known involution of the double vector bundle $TTM$ (\cite{Mac}, p. 363). For
$n=1$, we have an isomorphism $T\jetb{m}{M}\simeq\jetb{m}{TM}$ whose dual version,
$T^*\jetb{m}{M}\simeq\jetb{m}{T^*M}$ has been studied in \cite{CCSS}.

Let $(U, (x_a))$ be a chart on $M$. Let us define the functions $x_a^{(r, s)}\in \Csi(q^{-1}(U))$ by
\[
x_a^{(r, s)}([\za]_{m, n}) := \left. \frac{d^{r+s}}{dt^{r}du^s}\right| _{t=0, s=0} x_a (\za(t, u)).
\]
The functions $x_a^{(r, s)}$, $0\leq r\leq m$, $0\leq s\leq n$,  form a homogeneous coordinate system on
$q^{-1}(U)$. The degree of $x_a^{(r, s)}$ is $(r, s)$.
\end{example}

\begin{example} Let $h:\Real_{\ge 0}\times F\to F$ be a homogeneity  structure of degree $m$ with a base $h_0(F)=F_0$. Let
$n\in\N$. Let us define the action $\jetb{n}h: \Real_{\ge 0}\times \jetb{n}F\to \jetb{n}F$ by $(\jetb{n}h)_t
:= \jetb{n}h_t$. It is easy to prove that $\jetb{n}h$ defines a second homogeneity structure on $T^nF$, this time of
degree $m$. Moreover, $\jetb{n}h$ commutes with the canonical homogeneity structure $h^{T^nF}$ of the $n$th
tangent bundle. Indeed, for a curve $\za:\Real\to F$ and $t, u\in\Real$, both $u.T^nh_t([\za]_n)$ and
$(T^nh_t)(u.[\za]_n)$ are equal to the class of the curve $s\mapsto h_t(\za(us))$. Hence, $T^nF$ is a DHS of
degree $(m, n)$:
\[
\xymatrix{
\jetb{n} F \ar[rr]^{h^{T^nF}_0} \ar[d]_{T^nh_0} && F \ar[d]^{h_0}\\
\jetb{n} F_0 \ar[rr]_{{h}^{T^nF_0}_0} &&  F_0 }
\]
We shall call $\jetb{n}h$ the $n$th {\it tangent prolongation} of the homogeneity structure $h$ on  $F$.
\end{example}

\begin{example} Let $(F, h^1, h^2)$ be a DHS of degree $(m, n)$ and let $k\in \N$. Then the $k$-th tangent
prolongations $T^kh^1$ and $T^kh^2$ commute:
\[
(T^kh^1)_t\circ (T^kh^2)_u = T^k(h^1_t\circ h^2_u) = T^k(h^2_u\circ h^1_t) = (T^kh^2)_u\circ (T^kh^1)_t.
\]
Hence, we obtain three commuting homogeneity structures on $T^kF$: $T^kh^1$, $T^kh^2$, and the canonical
homogeneity structure of $T^kF$, which turn $T^kF$ into a triple graded bundle.
\end{example}

The main result of this  section says that any DHS is locally trivial in the sense that we can find an
atlas with local coordinates being simultaneously homogeneous with respect to both homogeneity structures.
\begin{theo}\label{thm:ltrv}  {\rm (On local triviality of DHS)}

\noindent Let $(F, h^1, h^2)$ be a double homogeneity structure of degree $(m, n)$. There exist a covering
$(U)_{U\in\mathcal{U}}$ of $M$ and homogeneous coordinate charts $(\pi^{-1}(U), \xars)_{0\leq r\leq m, 0\leq
s\leq n}$ whose coordinates $\xars$ are homogeneous of bi-degree $(r, s)\in\N\times\N$. Moreover, the
transition functions between the charts preserve this $\N\times \N$-grading, and $(F, h^1, h^2)$ can be naturally embedded as a subbundle in the double homogeneity structure $T^{m,n}F\simeq T^nT^mF\simeq T^mT^nF$.
\end{theo}

\noindent

In the proof of Theorem~\ref{thm:main} we found out that, for any homogeneity structure $h$ of degree $n$ on a manifold $M$, the tangent space $T_\zvy M$ at a point $\zvy\in M_0 = h_0(M)$ admits a canonical decomposition  into the direct
sum (compare with (\ref{e:TMdec})):
\be{e:dec}
T_\zvy M = E_0\oplus E_1 \oplus \ldots \oplus  E_n,
\ee
where $E_0 = T_\zvy M_0$ and $E_r$ is the image of the projection
\be{e:Qrr}
\kdiff{r} H_t : T_\zvy M\to T_\zvy M,
\ee
and $H_t$ is defined by (\ref{e:Htt}). Note that the domain of the map $Q_r$ defined by (\ref{e:Qr}) is only
$E_r \oplus \ldots \oplus  E_n$. We prefer to work with the projection (\ref{e:Qrr}) whose domain is the whole
tangent space $T_\zvy M$ that coincides  with $Q_r$ on the domain of $Q_r$ and which we shall also denote by
$Q_r$. As $\zvy$ varies through $M_0$, we get the decomposition of the tangent bundle of $M$ restricted to
$M_0$, $(TM)_{|M_0} = \bigcup_{\zvy\in M_0} T_\zvy M$, into the direct sum of subbundles $E_0$, $E_1, \ldots,
E_n$. We have also proved there that if only the coordinate chart $(h_0^{-1}(U), (x_{r,a}))$ (not necessarily
homogeneous) on $M$ is such that $\frac{\partial}{\partial x_{r,a}}(\zvy)\in E_r$ for $1\leq r\leq n$ and
every $\zvy\in U\subset M_0$,  and $(x_{0, a})$ are coordinates on $U$ pulled back to  $h_0^{-1}(U)$, then in
the adapted coordinates $(x_{r, a}^{(s)})$ on the higher tangent bundle $\jetbndln$ the functions
$(x^{(r)}_{r, a |\widetilde{M}})$ form a homogeneous coordinate system on $\widetilde{M}$, where
$\widetilde{M} \subset\jetbndln$ is the image of the canonical embedding of $M$ into $\jetbndln$ respecting
the homogeneity structures. We are ready now to start the proof of our theorem.

\medskip
\bepf
We shall embed $F$ into the jet bundle $\jetbndlmn{F}$. Let us define $\zb(t, u) :=h_t^1\circ h_u^2 : F\to F$
and let $\zb_p(t, u) = \zb(t, u)(p)$, for $p\in F$. Let $\zvf: F\to \jetbndlmn{F}$ be defined by
\[
\zvf(p) = [\zb_p]_{m, n}.
\]
Observe that $\zvf$ respects the homogeneity structures. We have
\[
\zb_{h^1(s, p)}(t, u) = h_t^1(h_u^2(h_s^1(p))) = h_{st}^1(h_u^2(p)) = (s._1 \zb_p) (t, u),
\]
so that $\zvf(h^1(s, p)) = s._1\zvf(p)$ and, similarly, $\zvf(h^2(s, p)) = s._2\zvf(p)$. Let
\be{e:TFdec1}
T_\zvy F = E_0^1\oplus E_1^1 \oplus \ldots \oplus E_m^1
\ee
be the canonical decomposition (\ref{e:dec}) of the tangent space of $F$ at a point $\zvy$ in the base $F_1$.
Let us assume that $\zvy\in M$, i.e., $\zvy\in F_1\cap F_2$. Considering the second homogeneity structure on
$F$, we can write
%\be{e:TFdec2}
$$
T_\zvy F = E_0^2\oplus E_1^2 \oplus \ldots \oplus E_n^2.
$$
%\ee
We shall prove that
\be{e:TFdec}
T_\zvy F = \bigoplus_{r,s} E_{r,s},
\ee
where $E_{r,s} = E_r^1\cap E_s^2$, $0\leq r\leq m$, $0\leq s\leq n$.

Let $H_t^i: T_\zvy F\to T_\zvy F$, $H_t^i = \differ_\zvy h_t^i$, $i=1,2$. Differentiating  $h_u^2\circ h^1_t =
h^1_t\circ h^2_u: F\to F$ at $\zvy$, we get
\be{e:commHt}
H_u^2\circ H_t^1 =H_t^1\circ H_u^2.
\ee
Recall that $E_r^i$, $i=1,2$, is the image of the projection
$$Q_r^i = \kdiff{r} H_t^i : T_\zvy F\to T_\zvy F\,.
$$
From (\ref{e:commHt}) we get that $H_u^2(E_r^1)\subseteq E_r^1$, for $0\leq r\leq m$. Hence, also $Q_s^2$
respects the decomposition (\ref{e:TFdec1}). It follows that $E_r^1 = \bigoplus_s E_r^1\cap E_s^2$ and
(\ref{e:TFdec}), as we have claimed. Since $\zvy$ varies through $M$, we get the decomposition of $(TF)_{|M}$
into a direct sum of the vector subbundles $E_{r,s}\to M$.

Let us choose a coordinate chart $(\zp^{-1}(U), (\xars))$ for $\zp:F\to M$ in such a way that
$\frac{\partial}{\partial \xars}(\zvy)\in E_{r,s}$ for $\zvy\in U$. Assume also that the coordinates $ x_{(r,
0), a }$, $x_{(0, s),b}$, $x_{(0,0),c}$ are constant on the fibers of $h^1_0$, $h^2_0$,  and $\zp= h^1_0 \circ
h^2_0$, respectively. We claim that in the adapted coordinates $(\xars^{(r', s')})$ on $\jetbndlmn{F}$ the
restrictions of $\xars^{(r, s)}$  form a homogeneous coordinate system on $\widetilde{F}$.

A general idea is to decompose $\zvf$ into the following embeddings:
\[
\xymatrix{
& F_2 \ar[dd]& & & \jetb{m}F_2\ar[dd] & & & \jetb{m}F\ar[dd] \\
F\ar[ru]\ar[dd]\ar[rrr]^{\zvf^1}&&&\jetb{m}F\ar[rrr]^{\tilde{\zvf}^2}\ar[dd]\ar[ru]&&&**[r]\jetb{n}\jetb{m}F\simeq \jetb{m,n}F \ar[dd]\ar[ru]& \\
&M&&&F_2 &&& F \\
F_1\ar[ur]&&&F\ar[ru]&&&\jetb{n}F\ar[ur] }
\]
Here, $\zvf^1$ and $\widetilde{\zvf}^2$ are the canonical embeddings (\ref{e:defzvf}) of the homogeneity
structures $h^1$ and the tangent prolongation  $\widetilde{h}^2:= \jetb{m}h^2$ of $h^2$, respectively, into
the corresponding higher tangent bundles.

Note that $\zvf^1$, and for the same reason $\widetilde{\zvf}^2$, respects both homogeneity structures.
Indeed, $\zvf^1(h_t^2(p))$, $p\in F$, is the class of the curve $u\mapsto h_u^1(h_t^2(p))$ in $\jetb{m}F$.
This coincides with
$$\widetilde{h}^2_t(\zvf^1(p))
= (\jetb{m}h_t^2)\left([u\mapsto h_u^1(p)]_m\right) = [u\mapsto h_t^2(h_u^1(p))]_m\,.$$ Let us consider now
$\widetilde{F}^1 := \zvf^1(F)$ -- the graded subbundle of degree $m$ of $\jetb{m}{F}\to F$ with the base
$F_1$. Since $\frac{\partial}{\partial\xars}$ is tangent to $E_r^1$ at points of the base $F_1$, we conclude
that $(\xars^{(r)})$ form a homogeneous coordinate system on $\widetilde{F}^1$, where $(\xars^{(r')})$ are the
adapted coordinates on $\jetb{m}{F}\to F$. Moreover, the decomposition (\ref{e:dec}) applied to
$\widetilde{F}^1$ with respect to the second homogeneity structure $\widetilde{h}^2_t$ on $\jetb{m}{F}$ reads
\be{e:decTF1}
T_\zvy\widetilde{F}^1 = \bigoplus_{s=0}^n (T_\zvy\zvf^1)(E_s^2).
\ee
Indeed, by differentiating $\zvf^1\circ h_t^2 = \widetilde{h}^2_t\circ\zvf^1$ at $\zvy$, we get
$(T_\zvy\zvf^1)\circ H_t^2 = \widetilde{H}_t^2\circ (T_\zvy\zvf^1)$, where $\widetilde{H}_t^2 =
\differ_\zvy\widetilde{h}_t^2\in\text{End}(T_\zvy\widetilde{F}^1)$. Hence, $(T_\zvy\zvf^1)\circ Q_s^2 =
\widetilde{Q}_s^2\circ (T_\zvy\zvf^1)$, where $\widetilde{Q}_s^2 = \kdiff{s} \widetilde{H}_t^2$. Since $E_s^2$
is the image of the projection $Q_s^2$, the map $\widetilde{Q}_s^2$ is the projection on
$(T_\zvy\zvf^1)(E_s^2)$, and (\ref{e:decTF1}) follows.

\medskip\noindent
Finally, as $\frac{\partial}{\partial \xars} (\zvy) \in E_s^2$ and
$$(T_\zvy\zvf^1)\left(\frac{\partial}{\partial \xars} (\zvy)\right) =  \frac{\partial}{\partial
\xars^{(r)}} (\zvy) \in T_\zvy \jetb{m}{F}\,,$$ we get
\[
\frac{\partial}{\partial \xars^{(r)}} (\zvy) \in (T_\zvy\zvf^1)(E_s^2) \subset T_\zvy\widetilde{F}^1.
\]
Hence, in the adapted coordinates $(\xars^{(r')})^{(s')}$ on $\jetb{n}{\jetb{m}{F}}$, the functions
$(\xars^{(r)})^{(s)}$ are homogeneous coordinates on $\widetilde{\zvf}^2(\widetilde{F}^1)$. Moreover,
$\widetilde{\zvf}^2(\widetilde{F}^1)$ is identified with $\widetilde{F}=\zvf(F)$ by means of the isomorphism
$I$ (\ref{e:jetiso}), and $I^*((\xars^{(r)})^{(s)}) = \xars^{(r, s)}$. This proves our claim and completes the
proof of the theorem, as $I^{-1}\circ \widetilde{\zvf^2}\circ \zvf^1: F\to \widetilde{F}$ intertwines the
double homogeneity structures on $F$ and $\widetilde{F}$.

\epf

\medskip
One can obtain, quite in parallel to the above proof for a DHS, a similar result for a general $r$-tuple
homogeneity structure.

\begin{theo} Let $(F, h^1, \dots,h^r)$ be a $r$-tuple homogeneity structure of degree $(m_1,\dots,m_r)$. There exist a covering $(U)_{U\in\mathcal{U}}$ of $M=\zp(F)$, $\zp=h^1_0\circ\cdots\circ h^r_0$, and
homogeneous coordinate charts
$$\left(\pi^{-1}(U), x_{(k_1,\dots,k_r)}\right), \ {0\leq k_1,\leq m_1,\dots, 0\leq k_r\leq m_r}\,,$$ whose coordinates
$x_{(k_1,\dots,k_r)}$ are homogeneous of degree $(k_1,\dots,k_r)\in\N^r$. Moreover, the transition functions
between the charts preserves this $\N^r$-grading, and $(F, h^1, \dots,h^r)$ can be naturally embedded as a subbundle in the $r$-tuple homogeneity structure $T^{m_1,\dots,m_r}F\simeq T^{m_1}(\dots (T^{m_r} F)\dots)$.
\end{theo}

\medskip
\begin{remark} In the case of a double homogeneity structure $(F, h^1, h^2)$ of degree $(1,1)$, and so a double vector bundle, the above theorem provides an alternative proof of a local decomposition theorem for a double vector bundle. Any double vector bundle $(F, h^1, h^2)$ can be canonically embedded into double vector bundle $\jetb{1,1}{F} \simeq (T(TF); TF, TF; F)$. The adapted coordinates on $TTF$ give, by restrictions,  homogeneous coordinates on $\widetilde{F}$ (the isomorphic image of $F$ in $TTF$).
More generally, for an $r$-tuple homogeneity structures, if each of these homogeneity structures is of degree 1, so each corresponds to a vector bundle structure, we obtain an $r$-vector bundle in the terminology of \cite{GR}. Hence, any $r$-vector bundle $F$ can be canonically embedded as an $r$-vector subbundle of $\jetb{1,1,\ldots, 1}{F} \simeq T(T(\ldots
(TF)\ldots)$. It follows that an $r$-vector bundle $F$ decomposes locally (although not canonically) as
$$
F \simeq U\times \prod_{\ze\in\{0,1\}^r, \ze\neq 0} V_\ze
$$
for some open set $U$ and vector spaces $V_\ze$.
\end{remark}
\medskip
\begin{remark} Note that if $\zD_F^i$, $i=1,2$, are weight vector fields (\ref{wvf}) on $F$ defining a double homogeneity structure on $F$ of degree $(m, n)$, then $\zD_F := \zD_F^1 + \zD_F^2$ defines a homogeneity structure on $F$ of degree $m+n$. For example, with a double vector bundle one can associate a homogeneity structure of degree $2$ like in Example (\ref{e:2vb}). In other words, commuting weight vector fields, unlike the
Euler vector fields associated with vector bundles, are closed with respect to addition.
\end{remark}
\section{Acknowledgements}
We are indebted to the referee for extremely useful remarks and comments that led to improvement of our presentation.

\bigskip
\noindent Janusz Grabowski\\ Institute of Mathematics, Polish Academy of Sciences\\\'Sniadeckich 8, P.O. Box
21, 00-956 Warszawa,
Poland\\{\tt jagrab@impan.pl}\\\\
\noindent Miko\l aj Rotkiewicz\\
Institute of Mathematics,
                University of Warsaw \\
                Banacha 2, 02-097 Warszawa, Poland \\
                 {\tt mrotkiew@mimuw.edu.pl}
\end{document}